\documentclass[1pt]{article}
\usepackage{amsmath, amssymb, amsthm, amsfonts, cases}
\usepackage{mathrsfs}
\usepackage{url}
\usepackage{authblk}
\usepackage[usenames]{color}
\usepackage{geometry}

\newtheorem{rmk}{Remark}[section]
\newtheorem{lemma}{Lemma}[section]
\newtheorem{theorem}{Theorem}[section]

\newtheorem{definition}{Definition}[section]
\newtheorem{pro}{Problem}[section]
\newtheorem{ass}{Assumption}[section]

\newcommand{\eps}{\varepsilon}

\newcommand{\R}{\mathbb{R}}

\allowdisplaybreaks

 \makeatother
\begin{document}

\title{  Partial Information Near-Optimal Control of Forward-Backward Stochastic Differential System with Observation Noise
 \thanks{This work was supported by the Natural Science Foundation of Zhejiang Province
for Distinguished Young Scholar  (No.LR15A010001),  and the National Natural
Science Foundation of China (No.11471079, 11301177) }}

\date{}

   \author{ Qingxin Meng\thanks{Corresponding author.   E-mail: mqx@zjhu.edu.cn}  \hspace{1cm}
   Qiuhong Shi
 \hspace{1cm}  Maoning Tang
\hspace{1cm}
\\\small{Department of Mathematics, Huzhou University, Zhejiang 313000, China}}

\maketitle
\begin{abstract}

This paper first
makes an attempt to investigate  the
partial information
near optimal control of  systems governed by forward-backward stochastic differential equations with
observation noise under the assumption of a convex control domain.
By Ekeland's variational principle and some basic estimates for state processes
and adjoint processes, we establish the necessary conditions for any $\varepsilon $-near
optimal control in a local form with an error order of exact $\varepsilon ^{%
\frac{1}{2}}.$  Moreover, under additional convexity conditions on Hamiltonian
function, we prove that an $\varepsilon $-maximum condition in terms of the
Hamiltonian in the integral form is sufficient for near-optimality.
\end{abstract}

\textbf{Keywords} Near Optimal Control, Forward-Backward
Stochastic Differential Equation, Partial Information,

\maketitle

\section{ Introduction}

In recent years, near-optimization has become an important research topic in optimal control theory.
Compared with its exact-optimality counterpart, near-optimality has many appealing properties, which are
useful in both theory and applications. For example, near-optimal controls always exist while optimal controls
may not exist in many situations; there are many candidates for near-optimal controls which can be selected
easily and appropriately for analysis and implementation; in most practical situations, a near-optimal
control suffices to guide decision making while it is usually unrealistic and unnecessary to explore
optimal controls, which are very sensitive to external perturbation. Interested readers may refer to
\cite{zhou1998stochastic} for more discussion about the merits of near-optimality.

Indeed, there has been a large pile of literature on near-optimal controls in both deterministic
and stochastic cases. \cite{zhou1995deterministic,zhou1996deterministic}
investigated near-optimal controls for deterministic
dynamical systems. The history of near-optimality under stochastic systems can be dated back to
\cite{elliott1980the}, where necessary conditions were derived for some near-optimal
controls. \cite{zhou1994a} provided a sufficient condition for near-optimal stochastic
controls and applied it to general manufacturing systems. \cite{zhou1998stochastic} derived necessary and
sufficient conditions for all near-optimal controls under forward systems of the diffusion type.
Current research focuses on near-optimal controls under various systems. Please see \cite{liu2005near}
for regime-switching systems, \cite{bahlali2009necessary}, \cite{huang2010near} and \cite{zhang2012stochastic}
for forward-backward systems, \cite{chighoub2011near} and \cite{hafayed2012maximum} for jump-diffusion systems,
\cite{hui2011near} for recursive systems, and references therein.
 Recently, Meng and Shen
  \cite{MeSh} revisits the stochastic near-optimal control problem considered by
Zhou \cite{zhou1998stochastic}, where the stochastic system is given by a controlled stochastic
differential equation with the control variable taking values in a general control space and
entering both the drift and diffusion coefficients an  improve
the error bound of order from ``almost" $\varepsilon^\frac{1}{3}$ in \cite{zhou1998stochastic}
to ``exactly" $\varepsilon^\frac{1}{3}$.

 Meanwhile, there have been growing interests on stochastic
optimal control problems under partial information, partly due to
the applications in mathematical finance. For the partial
information  optimal control problem, the objective is to find an
optimal control for which the controller has less information than
the complete information filtration. In particular, sometimes an
economic model in which there are information gaps among economic
agents can be formulated as a partial information optimal control
problem (see ${\O}$ksendal \cite{okb}, Kohlmann and Xiong \cite{KoX}).
 Baghery and
${\O}$ksendal \cite{BOK} established a maximum principle of forward systems
with jumps  under partial information.
In 2009,  Meng \cite{Meng} studied
 a partial information stochastic optimal control problem of continuous fully
coupled forward-backward stochastic systems driven by a Brownian
motion.
In 2013, Wang, Wu and Xiong \cite{WWX} studied a partial information optimal control problem derived by forward-backward stochastic systems with correlated noises between the system and the observation. Utilizing a direct method, an approximation method, and a Malliavin derivative method, they established three versions of maximum principle (i.e., necessary condition) for optimal control. In 2017, Meng, Shi and Tang \cite{MST}  revisits the partial information optimal control problem  considered by Wang, Wu and Xiong \cite{WWX} where they improve
the $L^p-$ bounds on the control  from $L^8-$ bounds
to the following  $L^4-$ bounds.

 The purpose of the present paper
is to make a first attempt to discuss the
partial information near optimal control for forward backward stochastic
differential systems with observation noise. Its  main contribution
is the developments of  maximum principle and verification theorem  of  the
partial information
near optimal control in a uniform manner by Ekeland's variational
principle.

The paper is organized as follows. In section 2, we formulate the problem and give
various assumptions used throughout the paper. Section 3 is devoted to derive necessary as well as sufficient near optimality
conditions in the form of stochastic maximum principles in a unified
way.

\section{Formulation of Problem}

In this section, we introduce
some basic notations  which will be
used in this paper.
Let ${\cal T} : = [0, T]$ denote a finite time index, where $0<T <
\infty$. We consider a complete probability space $( \Omega,
{\mathscr F}, {\mathbb P} )$  equipped with two one-dimensional
standard Brownian motions $\{W(t), t \in {\cal T}\}$ and $\{Y(t),t \in {\cal T}\},$
respectively. Let $%
\{\mathscr{F}^W_t\}_{t\in {\cal T}}$ and $%
\{\mathscr{F}^Y_t\}_{t\in {\cal T}}$ be $\mathbb P$-completed natural
filtration generated by $\{W(t), t\in {\cal T}\}$ and $\{Y(t), t\in {\cal T}\},$ respectively. Set $\{\mathscr{F}_t\}_{t\in {\cal T}}:=\{\mathscr{F}^W_t\}_{t\in {\cal T}}\bigvee
\{\mathscr{F}^Y_t\}_{t\in {\cal T}}, \mathscr F=\mathscr F_T.$   Denote by $\mathbb E[\cdot]$ the expectation
under the probablity $\mathbb P.$
 Let $E$ be a Euclidean space. The inner product in $E$ is denoted by
$\langle\cdot, \cdot\rangle,$ and the norm in $ E$ is denoted by $|\cdot|.$
Let $A^{\top }$ denote the
transpose of the matrix or vector $A.$
For a
function $\psi:\mathbb R^n\longrightarrow \mathbb R,$ denote by
$\psi_x$ its gradient. If $\psi: \mathbb R^n\longrightarrow \mathbb R^k$ (with
$k\geq 2),$ then $\psi_x=(\frac{\partial \phi_i}{\partial x_j})$ is
the corresponding $k\times n$-Jacobian matrix. By $\mathscr{P}$ we
denote the
predictable $\sigma$ field on $\Omega\times [0, T]$ and by $\mathscr %
B(\Lambda)$ the Borel $\sigma$-algebra of any topological space
$\Lambda.$ In the follows, $K$ represents a generic constant, which
can be different from line to line.
Next we introduce some spaces of random variable and stochastic
 processes.
For any $\alpha, \beta\in [1,\infty),$
denote by
$M_{\mathscr{F}}^\beta(0,T;E)$ the space of all $E$-valued and ${%
\mathscr{F}}_t$-adapted processes $f=\{f(t,\omega),\ (t,\omega)\in \cal T
\times\Omega\}$ satisfying
$
\|f\|_{M_{\mathscr{F}}^\beta(0,T;E)}\triangleq{\left (\mathbb E\bigg[\displaystyle%
\int_0^T|f(t)|^ \beta dt\bigg]\right)^{\frac{1}{\beta}}}<\infty, $ by $S_{\mathscr{F}}^\beta (0,T;E)$ the space of all $E$-valued and ${%
\mathscr{F}}_t$-adapted c\`{a}dl\`{a}g processes $f=\{f(t,\omega),\
(t,\omega)\in {\cal T}\times\Omega\}$ satisfying $
\|f\|_{S_{\mathscr{F}}^\beta(0,T;E)}\triangleq{\left (\mathbb E\bigg[\displaystyle\sup_{t\in {\cal T}}|f(t)|^\beta \bigg]\right)^{\frac
{1}{\beta}}}<+\infty,$
by $L^\beta (\Omega,{\mathscr{F}},P;E)$ the space of all
$E$-valued random variables $\xi$ on $(\Omega,{\mathscr{F}},P)$
satisfying
$ \|\xi\|_{L^\beta(\Omega,{\mathscr{F}},P;E)}
\triangleq
\sqrt{\mathbb E|\xi|^\beta}<\infty,$
by $M_{\mathscr{F}}^\beta(0,T;L^\alpha (0,T; E))$ the space of all $L^\alpha (0,T; E)$-valued and ${%
\mathscr{F}}_t$-adapted processes $f=\{f(t,\omega),\ (t,\omega)\in[0,T]%
\times\Omega\}$ satisfying $
\|f\|_{\alpha,\beta}\triangleq{\left\{\mathbb E\bigg[\left(\displaystyle
\int_0^T|f(t)|^\alpha
dt\right)^{
\frac{\beta}{\alpha}}\bigg]\right\}^
{\frac{1}{\beta}}}<\infty. $
 Finally, we define the space
\begin{equation*}
{\ \mathbb{M}}^\beta[0,T]
:=S_{\mathscr{F}}^\beta(0,T;\mathbb R^n)\times S_{\mathscr{F}
}^\beta(0,T;\mathbb R^m) \times M_{\mathscr{F}}^\beta(0,T;L^\alpha (0,T; \mathbb R^m))
\times M_{\mathscr{F}}^\beta(0,T;L^\alpha (0,T; \mathbb R^m)).
\end{equation*}
Then $\mathbb{M}^\beta[0,T]$ is a Banach space with respect to the norm $||\cdot||_{%
\mathbb{M}^2}$ given by
\begin{equation*}
||\Theta (\cdot)||_{\mathbb{M}^\beta}^\beta=
\mathbb E\bigg[\sup_{0\leq t\leq T}\vert
x(t)|^\beta\bigg]+
\mathbb E\bigg[\sup_{0\leq t\leq T}\vert y(t)|^\beta\bigg]+\mathbb E\bigg[\bigg\{\int_0^T\vert z_1(t)|^2dt\bigg\}^{\frac{\beta}{2}}\bigg]+\mathbb E\bigg[\bigg\{\int_0^T\vert z_2(t)|^2dt\bigg\}^{\frac{\beta}{2}}\bigg]
\end{equation*}
for $\Theta (\cdot)=(x(\cdot),y(\cdot), z_1(\cdot), z_2(\cdot))\in \mathbb{M}^\beta[0,T].$

Consider the following
forward-backward stochastic differential equation

\begin{eqnarray} \label{eq:1}
\left\{
\begin{aligned}
dx(t)=&b(t,x(t),u(t))dt+ \sigma_1(t, x(t), u(t)) dW(t)+\sigma _2(t, x(t), u(t)) dW^u(t),
\\
dy(t)=&f(t,x(t),y(t),z_1(t),z_2(t),u(t))dt+ z_1(t) dW(t)+  z_2(t) dW^u(t),\\
x(0)=&x,\\
y(T)=&\phi(x(T))
\end{aligned}
\right.
\end{eqnarray}
with one observation  processes $Y(\cdot)$ driven by the following stochastic differential equation

\begin{eqnarray}\label{eq:2}
\left\{
\begin{aligned}
dY(t)=&h(t,x(t),u(t))dt+ dW^u(t),\\
Y(t)=& 0,
\end{aligned}
\right.
\end{eqnarray}
where $b: {\cal T} \times \Omega \times {\mathbb R}^n
 \times U  \rightarrow {\mathbb R}^n$,
  $\sigma_1:
{\cal T} \times \Omega \times {\mathbb R}^n
 \times U \rightarrow {\mathbb R}^n$, $\sigma_2: {\cal T} \times \Omega \times {\mathbb R}^n
 \times U \rightarrow {\mathbb R}^n $,
 $f: {\cal T} \times \Omega \times {\mathbb R}^n \times {\mathbb R}^m\times {\mathbb R}^m \times {\mathbb R}^m
 \times U \rightarrow {\mathbb R}^m,
 \phi: \Omega \times {\mathbb R}^n  \rightarrow {\mathbb R^m}$
 and $h: {\cal T} \times \Omega \times {\mathbb R}^n
 \times U \rightarrow {\mathbb R}$  are given random mapping with
 $U$ being a nonempty convex  compact  subset of $\mathbb R^k.$ In the above equations, $u(\cdot)$ is our admissible control
 process defined as follows.

\begin{definition}
   An admissible control process is defined as an ${\mathscr F}^Y_t$-adapted process valued in an nonempty
  convex compact subset $U$  in  $\mathbb R^K.$
  The set of all admissible controls is denoted by $\cal A.$
\end{definition}

Now we make the following standard  assumptions
 on the coefficients of the equations
 \eqref{eq:1} and
 \eqref{eq:2}.

\begin{ass}\label{ass:1.1}
(i)The coefficients $b$, $\sigma_1,\sigma_2 $ and $h$  are ${\mathscr P} \otimes {\mathscr
B} ({\mathbb R}^n)  \otimes {\mathscr B}
(U)$-measurable, $f$ is  ${\mathscr P} \otimes {\mathscr
B} ({\mathbb R}^n)\otimes {\mathscr
B} ({\mathbb R}^m)\otimes {\mathscr
B} ({\mathbb R}^m) \otimes {\mathscr
B} ({\mathbb R}^m)  \otimes {\mathscr B}
(U) $-measurable, and $\phi$ is ${\mathscr F}_T \otimes {\mathscr B} ({\mathbb
R}^n) $-measurable. For each $(x,
u) \in \mathbb {R}^n\times U$, $b (\cdot, x,u), \sigma_1(\cdot,x,u)$,
 $\sigma_2 (\cdot, x, u)$
and $h(\cdot, x,u)$ are all $\{\mathscr{F}_t\}_{t\in \cal T}$-adapted processes. For almost all $(t, \omega)\in {\cal T}
\times \Omega$, the mappings
\begin{eqnarray*}
(x,u) \rightarrow \psi(t,\omega,x, u)
\end{eqnarray*}
\begin{eqnarray*}
(x,y,z_1, z_2,u) \rightarrow l(t,\omega,x, y, z_1,z_2, u),
\end{eqnarray*}
\begin{eqnarray*}
 x\rightarrow \phi(\omega,x),
\end{eqnarray*}
are twice differentiable with all the partial derivatives of
$\psi,f$ and $\phi$ with respect to $(x, y, z_1, z_2, u)$ up to order $2$ being
continuous in $(x, y, z_1, z_2, u)$ and being uniformly
bounded, where $\psi=b, \sigma_1,  \sigma_2$ and $h.$
Moreover, $\sigma_2$ and $h$ are
uniformly bounded.
\end{ass}

Now we
begin to discuss  the well- posedness
of \eqref{eq:1} and \eqref{eq:2}.
Indeed, putting \eqref{eq:2} into the state equation \eqref{eq:1}, we get that
\begin{eqnarray} \label{eq:4}
\left\{
\begin{aligned}
dx(t)=&(b-\sigma_2h)(t,x(t),u(t))dt+ \sigma_1(t, x(t), u(t)) dW(t)+\sigma _2(t, x(t), u(t)) dY(t),
\\
dy(t)=&(f(t,x(t),y(t),z_1(t),z_2(t),u(t))
-z_2(t)h(t,x(t),u(t)))dt+ z_1(t) dW(t)+  z_2(t) dY(t),\\
x(0)=&x,\\
y(T)=&\Phi(x(T)).
\end{aligned}
\right.
\end{eqnarray}
Under Assumption \ref{ass:1.1}, for any
admissible control $u(\cdot)\in \cal A,$
the  equation \eqref{eq:4} admits a unique strong
  solution $( x(\cdot),  y(\cdot), z_1(\cdot), z_2(\cdot))\in \mathbb{M}^\beta[0,T], \forall \beta\geq 2.$

For  the strong solution
 $( x^u(\cdot),  y^u(\cdot), z_1^u(\cdot), z_2^u(\cdot))$ of  the equation
 \eqref{eq:4} associated with
 any given admissible control
$u(\cdot)\in \cal A,$ we  introduce a process
\begin{eqnarray}\label{eq:7}
\begin{split}
  \rho^u(t)
  =\displaystyle
  \exp^{\bigg\{\displaystyle\int_0^t h(s, x^u(s), u(s))dY(s)
  -\int_0^t\frac{1}{2} h^2(s, x^u(s), u(s))ds\bigg\}},
  \end{split}
\end{eqnarray}
which is abviously the solution to the following  SDE
\begin{eqnarray} \label{eq:8}
  \left\{
\begin{aligned}
  d \rho^u(t)=&  \rho^u(t) h(s, x^u(s), u(s))dY(s)\\
  \rho^u(0)=&1.
\end{aligned}
\right.
\end{eqnarray}
Under Assumption \ref{ass:1.1},
$\rho^u(\cdot)$ is
an  $( \Omega,
{\mathscr F}, \{\mathscr{F}_t\}_{t\in {\cal T}}, {\mathbb P} )-$
martingale. Define a new probability measure $\mathbb P^u$ on $(\Omega, \mathscr F)$ by
\begin{eqnarray}
  d\mathbb P^u=\rho^u(T)d\mathbb P.
\end{eqnarray}
 Then from Girsanov's theorem and \eqref{eq:2}, $(W(\cdot),W^u(\cdot))$ is an
$\mathbb R^2$-valued standard Brownian motion defined in the new probability
space $(\Omega, \mathscr F, \{\mathscr{F}_t\}_{0\leq t\leq T},\mathbb P^u).$
So $(\mathbb P^u, x^u(\cdot), y^u(\cdot),
z^u_1(\cdot), z^u_2(\cdot), \rho^u(\cdot),
 W(\cdot), W^u(\cdot))$ is a weak
solution on $(\Omega, \mathscr F, \{\mathscr{F}_t\}_{t\in \cal
T})$ of  \eqref{eq:1} and
\eqref{eq:2}.

 The cost functional is given by
\begin{eqnarray}\label{eq:13}
  \begin{split}
    J(u(\cdot)=\mathbb E^u\bigg[\int_0^Tl(t,x(t),y(t),z_1(t),z_2(t), u(t))dt+ \Phi(x(T))+\gamma(y(0))\bigg].
  \end{split}
\end{eqnarray}
 where $\mathbb E^u$ denotes the expectation with respect to the
probability space $(\Omega, \mathscr F, \{\mathscr{F}_t\}_{0\leq
t\leq T},\mathbb P^u)$ and  $l:
{\cal T} \times \Omega \times {\mathbb R}^n \times {\mathbb R}^m
\times {\mathbb R}^m\times {\mathbb R}^m
 \times U \rightarrow {\mathbb R},$ $\Phi: \Omega \times {\mathbb R}^n  \rightarrow {\mathbb R}$
 and  $\gamma: \Omega \times {\mathbb R}^m \rightarrow {\mathbb R}$
  are given random mappings
satisfying  the following assumption:

 \begin{ass}\label{ass:1.2}
 $l$ is ${\mathscr P} \otimes {\mathscr
B} ({\mathbb R}^n) \otimes {\mathscr B} ({\mathbb R}^m)\otimes {\mathscr B} ({\mathbb R}^m) \otimes {\mathscr B} ({\mathbb R}^m)  \otimes {\mathscr B}
(U) $-measurable,  $\Phi$ is ${\mathscr F}_T \otimes {\mathscr B} ({\mathbb
R}^n) $-measurable, and $\gamma$ is ${
\mathscr F}_0 \otimes {\mathscr B} ({\mathbb
R}^n) $-measurable. For each $(x,y, z_1,z_2,
u) \in \mathbb {R}^n\times \mathbb R^m
\times \mathbb R^m\times \mathbb R^m\times U$, $f (\cdot, x, y,z_1,z_2, u)$ is  an ${\mathbb F}$-adapted process,  $\Phi(x)$
is an ${\mathscr F}_{T}$-measurable random variable, and  $\gamma(y)$
is an ${\mathscr F}_{0}$-measurable random variable. For almost all $(t, \omega)\in [0,T]
\times \Omega$, the mappings
\begin{eqnarray*}
(x,y,z_1, z_2,u) \rightarrow l(t,\omega,x, y, z_1,z_2, u),
\end{eqnarray*}
\begin{eqnarray*}
x \rightarrow \Phi(\omega,x)
\end{eqnarray*}
and
\begin{eqnarray*}
y\rightarrow \gamma(\omega,y)
\end{eqnarray*}
twice differentiable with all the partial derivatives of
$l, \Phi$ and $\gamma$ with respect to $(x, y, z_1, z_2, u)$ up to order $2$ being
continuous in $(x, y, z_1, z_2, u)$ and being uniformly
bounded.
\end{ass}
 Under  Assumption \ref{ass:1.1} and
 \ref{ass:1.2},
  it is easy to check that
the cost functional is well-defined.

Then we  can put forward the following partially observed optimal control problem  in its weak formulation,
 i.e., with changing
 the reference probability space $(\Omega, \mathscr F, \{\mathscr{F}_t\}_{0\leq
t\leq T},\mathbb P^u),$ as follows.

\begin{pro}
\label{pro:1.1}
\begin{equation*}  \label{eq:b7}
V(x)=\displaystyle\inf_{u(\cdot)
\in \cal A}J(u(\cdot)),
\end{equation*}
subject to the
 state equation \eqref{eq:1}, the
 observation equation \eqref{eq:2}
 and the cost functional \eqref{eq:13}.
\end{pro}
Here $V (x)$ refers to the value function of Problem \ref{pro:1.1}.
Obviously, according to  Bayes' formula,
 the cost functional \eqref{eq:13} can be rewritten as
\begin{equation}\label{eq:15}
\begin{split}
J(u(\cdot ))=& \mathbb E\displaystyle\bigg[%
\int_{0}^{T}\rho^u(t)l(t,x(t),y(t),z_1(t),z_2(t), u(t))dt +\rho^u(T)\Phi(x(T))
+\gamma(y(0))\bigg].
\end{split}
\end{equation}
  Therefore, we  can translate   Problem \ref{pro:1.1}
  into  the following  equivalent optimal control problem in
  its strong formulation, i.e., without changing the  reference
probability space $(\Omega, \mathscr F, \{\mathscr{F}_t\}_{0\leq t\leq T},\mathbb P),$
where $\rho^u(\cdot)$ will be regarded as
an additional  state process besides the
state process $(x^u(\cdot), y^u(\cdot),
z^u_1(\cdot), z^u_2(\cdot)).$

\begin{pro}
\label{pro:1.2}
\begin{equation*}  \label{eq:b7}
V(x)=\displaystyle\inf_{u(\cdot)
\in \cal A}J(u(\cdot)),
\end{equation*}
 subject to
 the cost functional \eqref{eq:15}
 and  the following
 state equation
\begin{equation}
\displaystyle\left\{
\begin{array}{lll}
dx(t)=&(b-\sigma_2h)(t,x(t),u(t))dt+ \sigma_1(t, x(t), u(t)) dW(t)+\sigma _2(t, x(t), u(t)) dY(t),
\\
dy(t)=&(f(t,x(t),y(t),z_1(t),z_2(t),u(t))
-z_2(t)h(t,x(t),u(t)))dt+ z_1(t) dW(t)+  z_2(t) dY(t),\\
d \rho^u(t)=&  \rho^u(t) h(s, x^u(s), u(s))dY(s),\\
  \rho^u(0)=&1,
\\
x(0)=&x,\\
y(T)=&\Phi(x(T)).
\end{array}%
\right.  \label{eq:3.7}
\end{equation}
\end{pro}
A control process $\bar{u}(\cdot)\in {\cal A}$ is called optimal, if it achieves the infimum of
$J(u(\cdot))$ over $\cal A$ and the corresponding state
process $(\bar x(\cdot),
\bar y(\cdot),
\bar z_1(\cdot), \bar z_2(\cdot),
\bar \rho(\cdot))$ is called the optimal
state process. Correspondingly $(\bar{u}(\cdot);\bar x(\cdot),
\bar y(\cdot),
\bar z_1(\cdot), \bar z_2(\cdot),
\bar \rho(\cdot))$ is called an optimal pair of Problem \ref{pro:1.2}.
\begin{rmk}
 The present formulation
of the partially observed optimal control problem is quite similar
to a completely observed optimal control problem; the only
difference lies in the admissible class $\cal A$ of controls.
\end{rmk}

Since the objective of this paper is to study near-optimal controls rather than exact-optimal ones,
we give the precise definitions of near-optimality as given in.

\begin{definition}
For a given $\varepsilon\geq 0$,  an admissible pair  $(u^\varepsilon(\cdot); x^\varepsilon(\cdot), y^\varepsilon(\cdot) ,z^\varepsilon(\cdot), \rho^\eps(\cdot))$
is called $\varepsilon$-optimal, if
\begin{eqnarray*}
|J(u^\varepsilon(\cdot))-V(x)|\leq \varepsilon.
\end{eqnarray*}
\end{definition}

\begin{definition}
Both a family of admissible control pairs $(X^\varepsilon(\cdot), u^\varepsilon(\cdot))$ parameterized by
$\varepsilon\geq 0$ and any element $(X^\varepsilon(\cdot), u^\varepsilon(\cdot))$, or simply $u^\varepsilon(\cdot)$,
in the family are called near-optimal if
\begin{eqnarray*}
|J(u^\varepsilon(\cdot))-V(x)|\leq r(\varepsilon).
\end{eqnarray*}
holds for sufficient small $\varepsilon$, where $r$ is a function of $\varepsilon$ satisfying $r(\varepsilon)
\rightarrow 0$ as $\varepsilon\rightarrow 0$. The estimate $r(\varepsilon)$
is called an error bound. If $r(\varepsilon)=C\varepsilon^\delta$ for some $\delta>0$ independent
of the constant $C$, then $u^\varepsilon(\cdot)$ is called near-optimal with order $\varepsilon^\delta$.
\end{definition}

Before we conclude this section, let us recall Ekeland's variational principle.

\begin{lemma}[Ekeland's principle, \cite{ekeland1974variational} ]\label{lem:3.1}
Let $(S,d)$ be a complete metric space and $\rho (\cdot ):S \rightarrow {\mathbb R}$
be lower-semicontinuous and bounded from below. For $\varepsilon \geq 0$,
suppose $u^{\varepsilon }\in S$ satisfies
\begin{equation*}
\rho ( u^\varepsilon ) \leq \inf_{u\in S} \rho (u) + \varepsilon .
\end{equation*}
Then for any $\lambda >0$, there exists $u^\lambda \in S$ such that
\begin{eqnarray*}
\rho ( u^\lambda ) \leq \rho ( u^\varepsilon ) , \quad
d ( u^\lambda, u^\varepsilon  ) \leq \lambda ,
\end{eqnarray*}
and
\begin{eqnarray*}
\rho ( u^\lambda ) \leq \rho (u) + \frac{\varepsilon}{\lambda} d (u^\lambda, u) , {\mbox { for all }} u \in S.
\end{eqnarray*}
\end{lemma}

\section{Main Results}

In this section, we establish the necessary and sufficient conditions for a control to
be near-optimal. The proof of our main results is based on Ekeland's variational principle
and convex variation techniques as well as some delicate estimates for the state process and the adjoint processes.

To this end,
for the state equation \eqref{eq:3.7},
we first introduce the
corresponding adjoint equation.

Define the Hamiltonian
 ${\cal H}: \Omega \times {\cal T} \times \mathbb R^n \times \mathbb R^m \times \mathbb R^{m}
  \times \mathbb R^{m}\times U\times
\mathbb R^n\times \mathbb R^n
\times \mathbb R^n
  \times \mathbb R^{m}
  \times \mathbb R\rightarrow \mathbb R$  as follows:
\begin{eqnarray}\label{eq6}
&& {\cal H} (t, x, y, z_1, z_2,u,k, p,q_1, q_2, R_2) \nonumber \\
&& = l (t, x, y, z_1,z_2, u)
+ \langle b (t, x, u),  p \rangle
+ \langle \sigma_1(t, x, u), q_1\rangle + \langle \sigma_2 (t, x, u), q_2\rangle  +
\langle f(t,x,y,z_1,z_2,u), k\rangle  +\langle R_2, h(t,x,u)\rangle \ .
\end{eqnarray}

For any given admissible control pair $ = (u(\cdot); x^u(\cdot),
 y^u(\cdot), z^u_1(\cdot),
 z^u_2(\cdot),\rho^u(\cdot)), $ the corresponding adjoint equation is defined as follows.

\begin{numcases}{}\label{eq:18}
\begin{split}
dr^u(t)&=- l(t,\Theta^u(t), u(t))dt
 +R_1^u\left(
t\right)  dW\left(  t\right) +{ R}_{2}^u\left( t\right) dW^{
u}\left( t\right),
\\
dp^u\left(  t\right)
 &=-{\cal {H}}_{x}\left( t,\Theta^u(t), u(t),\Lambda^u(t),
  R^u_2(t)\right)dt
 +q_1^u
\left(  t\right)  dW\left(  t\right)  +{q}_{2}^u\left( t\right) dW^{
u}\left( t\right),
\\
dk^u\left(  t\right)  &=-{\cal {H}}_{y}\left(  t,\Theta^u(t), u(t),
\Lambda^u(t), R^u_2(t) \right)dt-
{\cal {H}}_{z_1}\left(  t,\Theta^u(t), u(t),\Lambda^u(t),R^u_2(t) \right)  dW\left(  t\right)
  \\&\quad\quad-{\cal {H}}_{z_2}\left(  t,\Theta^u(t), u(t),\Lambda^u(t),
  R^u_2(t) \right)dW^{
u}\left( t\right),
\\ p^u(T)&=\Phi_x(x^u(T))-\phi_x^*( x^u(T)) k(T),
\\ r^u(T)&=\Phi(x^u(T)),
\\ k^u(0)&=-\gamma_y(y^u(0)).
\end{split}
\end{numcases}
Here we have used the following short hand notation
\begin{eqnarray}
  \begin{split}
   &\Theta^u(t):=(x^u(t), y^u(t), z_1^u(t),z_2^u(\cdot)),
\\&\Lambda^u(t):=(k^u(t), p^u(t), q^u_1(t), q_2^u(\cdot)),
\\&\Gamma^u(t):=(r^u(t), R^u_1(t), R_2^u(\cdot)).
  \end{split}
\end{eqnarray}

\begin{eqnarray}\label{eq:19}
\begin{split}
  &{\cal H}_a(t,\Theta^u(t),
 u(t), \Lambda^u(t), R^u_2(t))
 \\=&{\cal {H}}_{a}\left(  t,x^u(t),
 y^u(t), z_1^u(t),
 z_2^u(t), u(t), p^u(t), q_1^u(t), q_2^u(t),
  k^u(t), R_2^u(t)-\sigma_2^*(t, x^u(t), u(t))p^u(t)-(z_2^u) ^*(t)
  k(t) \right),
  \end{split}
\end{eqnarray}
where $a=x, y,z_1, z_2,u.$

Note the adjoint equation \eqref{eq:18}
is a forward-backward
stochastic differential equation whose solution
consists of  an 7-tuple process $( k(\cdot),p(\cdot),q_1(\cdot), { q}_2(\cdot), r(\cdot),R_1(\cdot), R_2(\cdot ) ).$
Under Assumptions \ref{ass:1.1} and
\ref{ass:1.2},  by Proposition 2.1 in
  \cite{Mou} and Lemma 2
  in \cite{Xu95} ,
it is easily to see that  the adjoint equation \eqref{eq:18} admits
a unique solution $(k(\cdot), p(\cdot), q_1(\cdot), { q}_2(\cdot),r(\cdot), R_1(\cdot),R_2(\cdot ) ),
$  also called the adjoint process corresponding
the admissible pair $({u}(\cdot); x(\cdot),
y(\cdot),
z_1(\cdot), z_2(\cdot),
\rho(\cdot))$. Particularly, we write $(k^u(\cdot), p^u(\cdot), q_1^u(\cdot), { q}_2^u(\cdot),r^u(\cdot), R_1^u(\cdot),\\R_2^u(\cdot ) )$ for
the adjoint processes associated with any admissible pair $({u}(\cdot); x^u(\cdot),
y^u(\cdot),
z_1^u(\cdot), z_2^u(\cdot),
\rho^u(\cdot))$, whenever we want to
emphasize the dependence of $(k(\cdot),p(\cdot), q_1(\cdot), { q}_2(\cdot),r(\cdot), R_1(\cdot),R_2(\cdot ) ).$

In order to apply  Ekeland's variational principle to obtain  our main result, we must define a
distance $d$ on \ the space of admissible controls s.t $\left(
\mathcal{A }, d\right) $ is  a complete metric space. For any given  $v\left(
\cdot \right) ,u\left( \cdot \right) \in \mathcal{A},$ we define
\begin{equation}  \label{eq:3.4}
d\left( v(\cdot),u(\cdot)\right) =\left[ E\int_{0}^{T}\left\vert v\left( r\right) -u\left(
r\right) \right\vert ^{2}ds\right] ^{\frac{1}{2}}.
\end{equation}

\begin{lemma} \label{lem:3.5}
  Let Assumptions \ref{ass:1.1}
   and \ref{ass:1.2} be satisfied.
    For any admissible pair $( u (\cdot); \Theta^u(\cdot ),\rho^u(\cdot))=(u(\cdot);x^u(t), y^u(t), z_1^u(t),z_2^u(\cdot),\\\rho^u(\cdot))$ and the corresponding adjoint process $(\Lambda^u(\cdot), \Gamma^u(\cdot))=(k^u(\cdot), p^u(\cdot), q_1^u(\cdot), { q}_2^u(\cdot),r^u(\cdot), R_1^u(\cdot), R_2^u(\cdot ) )$, there exists a constant $C > 0$
such that
\begin{eqnarray}\label{eq:3.5}
\mathbb {E} \bigg [ \sup_{0\leq t\leq T} | \Theta^u( t ) |^m \bigg ]+\mathbb {E} \bigg [ \sup_{0\leq t\leq T} | \rho^u( t ) |^m \bigg ] \leq C ,
\end{eqnarray}
\begin{eqnarray}\label{eq:16}
\mathbb {E} \bigg [ \bigg ( \int_0^T|R_2 ( t ) |^2 dt \bigg )^{\frac{m}{2}} \bigg ] \leq C ,
\end{eqnarray}
where $m \geq 2$.
\end{lemma}
\begin{proof}
  The proof can be directly obtained
  by  combining Proposition 2.1 in
  \cite{Mou} and Lemma 2
  in \cite{Xu95}.
\end{proof}

\begin{lemma} \label{lem:3.2}

Let Assumptions \ref{ass:1.1}
   and \ref{ass:1.2} be satisfied.
    For any admissible pairs $( u (\cdot); \Theta^u(\cdot ),\rho^u(\cdot))=(u(\cdot);x^u(t), y^u(t), z_1^u(t),z_2^u(\cdot),\\\rho^u(\cdot))$ and
    $( v (\cdot); \Theta^v(\cdot ))=(v(\cdot);x^v(t), y^v(t), z_1^v(t),z_2^v(\cdot),\rho^u(\cdot))$
    and the corresponding adjoint processes $(\Lambda^u(\cdot), \Gamma^u(\cdot))=(k^u(\cdot), p^u(\cdot), q_1^u(\cdot),\\ { q}_2^u(\cdot), r^u(\cdot), R_1^u(\cdot), R_2^u(\cdot ) )$
    and
    $(\Lambda^v(\cdot), \Gamma^v(\cdot))=(k^v(\cdot), p^v(\cdot), q_1^v(\cdot), { q}_2^v(\cdot),r^v(\cdot), R_1^v(\cdot), R_2^v(\cdot ) )$,
 there exists a constant $C > 0$ such that
\begin{equation}
\begin{split}
\|\Theta^u (\cdot)-\Theta^v(\cdot)\|
^2_{\mathbb{M}^2}+\|\Lambda^u (\cdot)-\Lambda^v(\cdot)\|
^2_{\mathbb{M}^2}+\mathbb E\bigg[\int_0^T|R^{u}_2 ( t ) - R^{v}_2 (t)
|^2 dt  \bigg ]
\leq Cd\left( u(\cdot
),v\left( \cdot \right) \right) ^{2}.
\end{split}%
\end{equation}
\end{lemma}

\begin{proof}
  The proof can be directly obtained
  by  combining Proposition 2.1 in
  \cite{Mou} and Lemma 2
  in \cite{Xu95}.
\end{proof}

Now we state our main result, which provides the necessary condition for a control to
be near-optimal with order $\varepsilon^\frac{1}{2}$.

\begin{theorem}

Let Assumptions \ref{ass:1.1}
   and \ref{ass:1.2} be satisfied.
    Let $ ( u^{\varepsilon } ( \cdot  );
\Theta^{\varepsilon } ( \cdot  ),\rho^\eps(\cdot)  )= (u^{\varepsilon } ( \cdot
 ); x^{\varepsilon } ( \cdot  ), y^{\varepsilon
} ( \cdot  ), z^{\varepsilon } _1( \cdot  ), z^{\varepsilon } _2( \cdot  ),\rho^\eps(\cdot)  ) $ be $%
\varepsilon $-optimal pair of problem
\ref{pro:1.2}.
Then for any given  $\varepsilon >0$,  there
is a
positive constant $C$  s.t.
\begin{equation}  \label{eq:4.99}
\mathbb E^\eps\bigg[\int_0^T {\cal H}_{u}\left( t,\Theta^{\varepsilon }\left( t \right)
,u^{\varepsilon }\left( t \right),
\Lambda ^{\epsilon}(t),R^{\eps}_2 \right)
\cdot \left( u(t) -u^{\varepsilon }\left( t\right) \right)\bigg] \geq
-C\varepsilon ^{\frac{1}{2}}, for~any~u(\cdot)\in \cal A,
\end{equation}
where $(\Lambda ^{\epsilon}(\cdot),\Gamma ^{\epsilon}(\cdot))=(k
^{\epsilon}(\cdot),
p^\varepsilon(\cdot),
q^\varepsilon_1(\cdot),q^\varepsilon_2(\cdot),
r
^{\epsilon}(\cdot),
R^\varepsilon_1(\cdot),R^\varepsilon_2(\cdot)
 )$ is
the adjoint process  corresponding to $\left(
u^{\varepsilon }\left( \cdot \right);\Theta^{\varepsilon }\left(
\cdot \right), \rho^\eps(\cdot)\right)$
and $\mathbb E^\eps$
denotes the expectation with respect to the
probability space $(\Omega, \mathscr F, \{\mathscr{F}_t\}_{0\leq
t\leq T},\mathbb P^{u^\eps}).$
\end{theorem}

\begin{proof}
By
Lemma \ref{lem:3.2} and Assumptions \ref{ass:1.1} and \ref{ass:1.2}, we can deduce
that $J(u(\cdot))$ is continuous on $\cal A$ with respect to the metric
 \eqref{eq:3.4}.
Using Ekeland's variational principle ( see Lemma \ref{lem:3.1}) with $\delta =\varepsilon ^{\frac{1}{%
2}},$ there exists an admissible pair $(\bar u^{\varepsilon }\left(
\cdot \right);\bar\Theta^{\varepsilon }\left( \cdot \right), \bar \rho^{\varepsilon }(\cdot) )=\left(
 \bar{u}^{\varepsilon }\left( \cdot \right);\bar{x}^{\varepsilon
}\left(\cdot \right) ,\bar y^{\varepsilon }\left( \cdot \right),
\bar z^{\varepsilon }_1\left( \cdot \right),\bar z^{\varepsilon }_2\left( \cdot \right), \bar \rho^{\varepsilon }(\cdot)\right) $ such that

\begin{equation}\label{eq:4.10}
d\left( u^{\varepsilon }\left( \cdot \right) ,\bar{u}^{\varepsilon }\left(
\cdot \right) \right) \leq \varepsilon ^{\frac{1}{2}}
\end{equation}
and
\begin{equation}\label{eq:4.11}
J\left( u\left( \cdot \right) \right) -J\left( \bar{u}^{\varepsilon }\left(
\cdot \right) \right) \geq -\varepsilon ^{\frac{1}{2}}d\left( u\left( \cdot
\right) ,\bar{u}^{\varepsilon }\left( \cdot \right) \right) ,\forall u\left(
\cdot \right) \in \mathcal{A}.\text{ }
\end{equation}
Now we define a convex perturbed control $u^{\varepsilon ,\delta}\left( \cdot
\right)$ of $\bar{u}^{\varepsilon }\left( \cdot \right) $ as%
\begin{equation*}
u^{\varepsilon ,\delta}\left( \cdot \right) =\bar{u}^{\varepsilon }(\cdot )+\delta(%
\bar{u}^{\varepsilon }(\cdot )-u\left( \cdot \right) ),
\end{equation*}
where $u\left( \cdot \right) \in \mathcal{A}$ is an arbitrary given admissible control and $0\leq \delta\leq 1 $.\\
Then by the variational formula
(35) in Meng, Shi and Tang
\cite{MST}, \eqref{eq:4.11}  and the fact
\begin{equation*}
d\left( u^{\varepsilon ,\delta}\left( \cdot \right) ,\bar{u}^{\varepsilon }\left(
\cdot \right) \right) \leq C\delta,
\end{equation*}
we have
\begin{equation}\label{eq:23}
\begin{array}{ll}
&\mathbb E^{\bar {u}^{\varepsilon }}\bigg[\displaystyle\int_{0}^{T}{\cal H}_{u}( t, \bar {\Theta}^{\varepsilon }(t),\bar {u}^{\varepsilon }(t) ,\bar\Lambda ^{\epsilon}(t),\bar R ^{\epsilon}_2(t) ) \cdot ( u(t) -\bar{u}^{\varepsilon }( t)) dt \bigg]\\
&=\lim_{\varepsilon \longrightarrow 0^{+}}\frac{J(u^{\varepsilon ,
\delta}\left(
\cdot\right) )-J(\bar{u}^{\varepsilon }(\cdot))}{\delta} \geq \lim_{\varepsilon \longrightarrow 0^{+}}\frac{-\varepsilon ^{\frac{1}{2}}d\left( u^{\varepsilon ,\delta}\left( \cdot\right) ,
\bar{u}^{\varepsilon}\left( \cdot\right) \right) }{\delta} \\
&\geq -C\varepsilon ^{\frac{1}{2}},
\end{array}
\end{equation}
where $(\bar {\Lambda}^{\epsilon},\bar {\Gamma}^{\epsilon})=(\bar k^{\epsilon}(\cdot),
\bar p^\varepsilon(\cdot),\bar q^\varepsilon_1(\cdot),\bar q^\varepsilon_2(\cdot),\bar r^\varepsilon(\cdot),\bar R^\varepsilon_1(\cdot),\bar R^\varepsilon_2(\cdot) )$ is the
adjoint process corresponding to $(\bar u^\varepsilon(\cdot); \bar \Theta^\varepsilon(\cdot),\bar \rho^\varepsilon(\cdot)).$\\
Now in order to obtain the optimal  condition \eqref{eq:4.99}, we now have to  estimate  the
following formula:
\begin{eqnarray}\label{eq:24}
\begin{split}
&I^{\varepsilon }_1 :=\mathbb E^{\bar u^\eps}\bigg[\int_{0}^{T}{\cal H}_{u}\left( t,\Theta^{{\varepsilon }}(t) ,u^\epsilon(t),\Lambda ^{\epsilon}(t),R_2 ^{\epsilon}(t)\right) \cdot \left( u\left(
t\right) -u^{\varepsilon }\left( t\right) \right) dt\bigg] \\
&~~~~~~~~~~-\mathbb E^{\bar u^\eps}\bigg[\int_{0}^{T}H_{u}\left( t,\bar \Theta^\epsilon(t) ,\bar u^\epsilon(t),\bar\Lambda ^{\epsilon}(t),\bar R_2 ^{\epsilon}(t)\right) \cdot \left( u\left(
t\right) -\bar{u}^{\varepsilon }\left( t\right) \right) dt\bigg]
\end{split}
\end{eqnarray}
and
\begin{eqnarray}\label{eq:25}
\begin{split}
&I^{\varepsilon }_2 :=\mathbb E^\eps\bigg[\int_{0}^{T}{\cal H}_{u}\left( t,\Theta^{{\varepsilon }}(t) ,u^\epsilon(t),\Lambda ^{\epsilon}(t),
R_2 ^{\epsilon}(t)\right) \cdot \left( u\left(
t\right) -u^{\varepsilon }\left( t\right) \right) dt \bigg]\\
&~~~~~~~~~~-\mathbb E^{\bar u^\eps}\bigg[\int_{0}^{T}{\cal H}_{u}\left( t,\Theta^{{\varepsilon }}(t) ,u^\epsilon(t),\Lambda ^{\epsilon}(t),R_2 ^{\epsilon}(t)\right) \cdot \left( u\left(
t\right) -u^{\varepsilon }\left( t\right) \right) dt \bigg].
\end{split}
\end{eqnarray}

From Lemmas \ref{lem:3.5}
and \ref{lem:3.2},
it is easy to get that
\begin{eqnarray} \label{eq:26}
  I_1^\eps\geq -C\eps^{\frac{1}{2}}
\end{eqnarray}
and
\begin{eqnarray}\label{eq:27}
  I_2^\eps\geq -C\eps^{\frac{1}{2}}.
\end{eqnarray}

Therefore  combining
\eqref{eq:23}, \eqref{eq:24}, \eqref{eq:25},\eqref{eq:26}
and \eqref{eq:27},
we have
\begin{eqnarray}
  \begin{split}
&\mathbb E^\eps\bigg[\int_{0}^{T}{\cal H}_{u}\left( t,\Theta^{\varepsilon }\left( t \right) ,u^{\varepsilon }\left( t \right), \Lambda ^{\epsilon}(t),R_2 ^{\epsilon}(t) \right) \cdot \left(
u(r) -u^{\varepsilon }\left( t\right) \right) dt\bigg]
\\&=\mathbb E^{\bar u^\eps}\bigg[\displaystyle\int_{0}^{T}
{\cal H}_{u}( t, \bar {\Theta}^{\varepsilon }(t),\bar {u}^{\varepsilon }(t) ,\bar\Lambda ^{\epsilon}(t),  \bar R_2 ^{\epsilon}(t) ) \cdot ( u(t) -\bar{u}^{\varepsilon }( t)) dt\bigg]+I_1^\varepsilon+I_2^\varepsilon
\\&\geq -C\varepsilon ^{\frac{1}{2}},
\end{split}
\end{eqnarray}
which implies \eqref{eq:4.99} holds.
The proof is complete.

\end{proof}

 Next, we prove that under an additional assumptions, the near-maximum condition on the Hamiltonian function is sufficient for near optimality of Problem \ref{pro:1.2} in the case when
the observation process is not
affected by the control process.
Suppose that
$$h(t,x,u)=h(t)$$ is an
$\mathscr F^Y_t-$ adapted bounded
process.
  Define a new probability measure $\mathbb Q$ on $(\Omega, \mathscr F)$ by
\begin{eqnarray}
  d\mathbb Q=\rho(T)d\mathbb P,
\end{eqnarray}
where
\begin{eqnarray} \label{eq:43}
  \left\{
\begin{aligned}
  d \rho(t)=&  \rho(t) h(s)dY(s)\\
  \rho(0)=&1.
\end{aligned}
\right.
\end{eqnarray}

\begin{theorem}{\bf [Sufficient Maximum Principle] } \label{thm:4.1}
 Let Assumptions \ref{ass:1.1}
  and \ref{ass:1.2} be satisfied.   Let $(u^\eps (\cdot);
  \Theta^\eps (\cdot))=(u^\eps (\cdot);
  x^\eps (\cdot),
  y^\eps(\cdot), z^\eps_1(\cdot),
  z_2^\eps (\cdot))$ be an admissible pair with $\phi(x)=\phi x,$
  where $\phi$ is $\mathscr F_T-$measurable bounded  random variable.
If the following conditions are satisfied,
\begin{enumerate}
\item[(i)]  $\Phi$ and $\gamma$ is convex in $x$ and $y,$ respectively,
\item[(ii)] the Hamiltonian ${\cal H}$ is convex in $(x, y, z_1, z_2,  u)$,
\item[(iii)]for some $\eps>0$ and
any $u(\cdot)\in \cal A,$
\begin{eqnarray}\label{eq:5.119}
&& \mathbb E^Q\bigg[\int_0^T{\cal H}_u ( t,\Theta^\eps (t),u^\eps(t), \Lambda^\eps(t), R_2^\eps(t))
\cdot(u(t)-u^\eps(t))dt\bigg]
\geq -C\eps^\lambda
\end{eqnarray}
\end{enumerate}where $( {\Lambda}^{\epsilon}(\cdot), {\Gamma}^{\epsilon}(\cdot))=( k^{\epsilon}(\cdot),
 p^\varepsilon(\cdot), q^\varepsilon_1(\cdot), q^\varepsilon_2(\cdot), r^\varepsilon(\cdot), R^\varepsilon_1(\cdot), R^\varepsilon_2(\cdot) )$ is the
adjoint process corresponding to $( u^\varepsilon(\cdot);  \Theta^\varepsilon(\cdot)).$
Then
\begin{equation}\label{eq:32}
J\left( u^{\varepsilon }\left( \cdot \right) \right) \leq \inf_{u\left(
\cdot \right) \in \mathcal{A}}J\left( u\left( \cdot \right) \right)
+C\varepsilon ^{\lambda},
\end{equation}
where $C$ is a constant independent of $\varepsilon .$
\end{theorem}

\begin{proof}
Let $(u(\cdot); \Theta(\cdot))=({u}(\cdot);x^u(\cdot),
 y^u(\cdot),
 z^u_1(\cdot),  z^u_2(\cdot))$  be an arbitrary admissible pair. By the variational formula
(35) in \cite{MeSh},
we can represent the difference $J ( u (\cdot) ) - J ( u^\eps(\cdot) )$ as follows
\begin{eqnarray}\label{eq:40}
&&J (u (\cdot)) - J ( u^\eps(\cdot))\nonumber
\\ &=& {\mathbb E} ^Q
\bigg[\int_0^T \bigg [ { \cal H}(t,\Theta (t),u(t), \Lambda^\eps(t),R_2^\eps(t)) - {\cal H} (t,\Theta^\eps (t),u^\eps(t), \Lambda^\eps(t),R_2^\eps(t)) \nonumber \\&& - \big < {\cal H}_x (t,\Theta^\eps (t),u^\eps(t), \Lambda^\eps(t),R_2^\eps(t)),x^u (t) -
x^\eps (t)
\big>\nonumber
\\&&- \big <{ \cal H}_y (t,\Theta^\eps (t),u^\eps(t), \Lambda^\eps(t),\R_2^\eps(t)),
y^u (t) - y^\eps (t)\big>\nonumber
\\&&- \big <{\cal H}_{z_1} (t, \Theta^\eps (t),u^\eps(t), \Lambda^\eps(t),R_2^\eps(t)),
 z^u_1 (t) - z^\eps_1 (t)\big>
 \nonumber
\\&&- \big <{\cal H}_{z_2} (t, \Theta^\eps (t),u^\eps(t), \Lambda^\eps(t)),
 z^u_2(t) - z^\eps_2 (t)\big>
 dt\bigg]
\nonumber \\
&& + {\mathbb E}^{Q} \big [ \Phi(x^u(T))
 - \Phi(x^\eps(T)) - \left <
x^u (T) - x^\eps (T), \Phi_x (
x^\eps(T))  \right
> \big ]
\nonumber \\
&& + {\mathbb E}
 \big [ \gamma(y^u(0))
 - \gamma (y^{\eps}(0)) - \left <
y^u (0)
- y^\eps (0), \gamma_y (y^\eps(0))  \right
> \big ].
\end{eqnarray}

By the convexity of ${\cal H}$, $\Phi$
and $\gamma$ (i.e. Conditions (i) and (ii)), we have
\begin{eqnarray}\label{eq:41}
 &&{ \cal H}(t, \Theta^u(t),
u(t), \Lambda^\eps_1(t), R^\eps_2(t)) - {\cal H} (t, \Theta^\eps(t),
u^\eps(t),
\Lambda^\eps(t),
R_2^\eps(t))\nonumber
\\ &\geq&  \big < {\cal H}_x (t,\Theta^\eps (t),u^\eps(t), \Lambda^\eps(t), R^\eps_2(t)),x^u (t) -
x^{\eps} (t)
\big>\nonumber
+ \big <{ \cal H}_y (t,\Theta^\eps (t),u^\eps(t), \Lambda^\eps(t),R^\eps_2(t)),
y^u (t) - y^{\eps} (t)\big>\nonumber
\\&&+ \big <{\cal H}_{z_1} (t,\Theta^\eps (t),u^\eps(t), \Lambda^\eps(t),R^\eps_2(t)),
 z^u_1 (t) -
z_1^{\eps} (t)\big>\nonumber
+ \big <{\cal H}_{z_2} (t,\Theta^\eps (t),u^\eps(t), \Lambda^\eps(t),R^\eps_2(t)),
 z^u_2 (t) -
z_2^{\eps} (t)\big>
\\&&+ \big <{\cal H}_{u} (t,\Theta^\eps (t),u^\eps(t), \Lambda^\eps(t),R^\eps_2(t)),u (t) -
u^\eps(t)\big>,
\end{eqnarray}

\begin{eqnarray}\label{eq:42}
 \Phi(x^u(T))
 - \Phi(x^\eps(T)) \geq \left <
x^u (T) - x^\eps (T), \Phi_x (
x^u(T))  \right
>
\end{eqnarray}
and
\begin{eqnarray}\label{eq:43}
 \gamma(y^u(0))
 - \gamma (y^{\eps}(0)) \geq \left <
y^u (0)
- y^\eps (0), \gamma_y (y^\eps(0))  \right
>
\end{eqnarray}
Putting \eqref{eq:41},\eqref{eq:42} and  \eqref{eq:5.119} into \eqref{eq:40},
we have
\begin{eqnarray}
J (u (\cdot)) - J (u^\eps (\cdot)) \geq
-C\eps^\lambda .
\end{eqnarray}
Due to the arbitrariness of $u (\cdot)$, we can conclude that
\eqref{eq:32} holds. The proof is completed.
\end{proof}
\section{Conclusion}
 we studied near-optimal controls under forward-backward  stochastic systems with partial information.
By Ekeland's variational principle,
convex variation techniques and some delicate estimates, we establish a Pontryagin type near-maximum principle as a necessary and
a sufficient condition for near-optimality.
The error bound in the necessary condition is of order ``exactly" $\varepsilon^\frac{1}{2}$.

\bibliographystyle{model1a-num-names}

\end{document}